\documentclass[11pt]{article}
\usepackage{kbordermatrix}

\usepackage{ifpdf}
\ifpdf\let\liang\relax\fi
\usepackage{mathtools}

\usepackage{graphicx}
\usepackage{amsmath,amsfonts,amsthm}
\usepackage{amssymb,latexsym}
\usepackage{fixmath}
\usepackage{mathrsfs,amsbsy}
\usepackage{dsfont}
\usepackage{enumerate}
\usepackage{algorithm,algorithmic}
\usepackage{kbordermatrix}
\usepackage{xcolor}

\def\wtd{\widetilde}
\def\what{\widehat}

\DeclareMathOperator{\diag}{diag}

\DeclareMathOperator*{\opt}{opt}
\DeclareMathOperator{\rank}{rank}

\DeclareMathOperator{\sign}{sign}

\DeclareMathOperator{\tr}{trace}
\DeclareMathOperator{\trace}{trace}

\DeclareMathOperator{\F}{F}
\DeclareMathOperator{\HH}{H}

\def\bbC{\mathbb{C}}

\def\bbR{\mathbb{R}}

\def\cR{{\cal R}}

\def\sss{\scriptscriptstyle}

\allowdisplaybreaks

\let\trace\relax
\DeclareMathOperator{\trace}{tr}

\ifdefined\liang
\def\bbC{\mathbb{C}}
\def\bbR{\mathbb{R}}
\usepackage{xcolor}

\else

\newtheorem{theorem}{Theorem}[section]
\newtheorem{lemma}{Lemma}[section]
\newtheorem{corollary}{Corollary}[section]

\theoremstyle{definition}

\numberwithin{equation}{section}
\numberwithin{figure}{section}
\numberwithin{table}{section}

\def\sss{\scriptstyle}

\fi

\title{On Generalizing Trace Minimization}

\author{
Xin Liang%
\thanks{Yau Mathematical Sciences Center, Tsinghua University, Beijing 100084, China.
	Email: {\tt liangxinslm@tsinghua.edu.cn}.
	Supported in part by the National Natural Science Foundation of China
	NSFC-11901340.  }
\and
Li Wang%
\thanks{Department of Mathematics,
University of Texas at Arlington, Arlington, TX 76019-0408, USA. Supported in part by NSF DMS-2009689. Email: {\tt li.wang@uta.edu.}}
\and
Lei-Hong Zhang%
\thanks{School of Mathematical Sciences and Institute of Computational Science, Soochow University, Suzhou 215006, Jiangsu, China. Email: {\tt longzlh@suda.edu.cn}.
             Supported in part by the National Natural Science Foundation of China
             NSFC-11671246 and NSFC-12071332.}
\and
Ren-Cang Li%
\thanks{Department of Mathematics,
University of Texas at Arlington, Arlington, TX 76019-0408, USA. Supported in part by NSF DMS-1719620 and DMS-2009689.
Email: {\tt rcli@uta.edu.}}
}

\date{
      January 29, 2021
}

\begin{document}

\maketitle

\begin{abstract}
Ky Fan's trace minimization principle is extended along the line of the
Brockett cost function $\trace(DX^{\HH}AX)$ in $X$ on the Stiefel manifold, where $D$ of an apt size is positive definite.
Specifically, we investigate $\inf_X\,\trace(DX^{\HH}AX)$ subject to $X^{\HH}BX=I_k$ or $J_k=\diag(\pm 1)$. We establish
conditions under which the infimum is finite and when it is finite, analytic solutions are
obtained in terms of the eigenvalues and eigenvectors of the matrix pencil $A-\lambda B$, where $B$ is possibly indefinite
and singular, and $D$ is also possibly indefinite.
\end{abstract}



\section{Introduction}
Quadratic optimization problems with matrix arguments are drawing tremendous attentions lately in data science, and
they often involve traces of certain quadratic forms, for example, the trace ratio maximization problem
from the linear discriminant analysis (LDA)
\cite{kocs:2011,ngbs:2010,zhln:2010,zhln:2013,zhli:2014a,zhli:2014b}, the correlation maximization
from the canonical correlation analysis (CCA) and its variants \cite{chlg:2013,cugh:2015,hors:1961,wazb:2020,zhwb:2020}.
Some formulations admit analytical solutions in terms of matrix eigenvalue/singular value decompositions, but most  don't.
Among those that do admit analytical solutions, the most well-known one is perhaps
Ky Fan's trace minimization principle
\cite[p.248]{hojo:2013}
\begin{equation}\label{eq:min-trace-Herm}
\min_{X^{\HH}X=I_k}\trace(X^{\HH}AX)=\sum_{i=1}^k\lambda_i,
\end{equation}
where $\tr(\,\cdot\,)$ is the trace of a square matrix, $I_k$ is the $k\times k$ identity matrix, and $A\in\bbC^{n\times n}$ is Hermitian and its eigenvalues are denoted by $\lambda_i$ ($i=1,2,\ldots,n$) and arranged in the ascending order:
\begin{equation}\label{eq:lambda-ascending}
	\lambda_1\le\lambda_2\le\cdots\le\lambda_n.
\end{equation}
Moreover for any minimizer
$X_{\min}$ of \eqref{eq:min-trace-Herm}, i.e., $\trace(X_{\min}^{\HH}AX_{\min})=\sum_{i=1}^k\lambda_i$, its columns span $A$'s invariant
subspace
associated with the first $k$ eigenvalues $\lambda_i$, $i=1,2,\ldots,k$.

The result \eqref{eq:min-trace-Herm} can be
straightforwardly extended to
\begin{equation}\label{eq:min-trace-Herm-generalized}
	\min_{X^{\HH}BX=I_k}\trace(X^{\HH}AX)=\sum_{i=1}^k\lambda_i,
\end{equation}
where $A,\, B\in\bbC^{n\times n}$ are Hermitian and $B$ is positive definite, and now $\lambda_i$ are the eigenvalues
of matrix pencil $A-\lambda B$ and ordered as in \eqref{eq:lambda-ascending}. Essentially, \eqref{eq:min-trace-Herm-generalized} is
no more general than \eqref{eq:min-trace-Herm}. In fact, upon substitutions: $X\leftarrow B^{1/2}X$ and
$A\leftarrow B^{-1/2}AB^{-1/2}$, \eqref{eq:min-trace-Herm-generalized} reduces to \eqref{eq:min-trace-Herm}. A nontrivial
extension of \eqref{eq:min-trace-Herm}, similar in form to \eqref{eq:min-trace-Herm-generalized}, is
\cite{kove:1995,lilb:2013}
\begin{equation}\label{eq:min-trace-Herm-SDP}
				\inf_{X^{\HH}BX=J_k}\trace(X^{\HH}AX)=\sum_{i=1}^{k_+}\lambda_i^+-\sum_{i=1}^{k_-}\lambda_i^-
\end{equation}
for a positive semi-definite matrix
pencil\footnote {$A,\, B\in\bbC^{n\times n}$ are Hermitian and there exists $\lambda_0\in{\mathbb R}$
     such that $A-\lambda_0B$ is positive semi-definite \cite[Definition 1.1]{lilb:2013}.}
$A-\lambda B$, where  $\lambda_i^{\pm}$ are finite eigenvalues
of $A-\lambda B$ and are arranged in the order as
$$ 
	\lambda_{n_-}^-\le\cdots\le\lambda_1^-\le\lambda_1^+\le\cdots\le\lambda_{n_+}^+,
$$ 
$n_-$ and $n_+$ are the numbers of negative and positive eigenvalues of $B$, respectively, and
\begin{equation}\label{eq:Jk-dfn}
0\le k_{\pm}\le n_{\pm},\,\,
k=k_++k_-,\,\,
	J_k=\begin{bmatrix}
		I_{k_+} &  \\
		& -I_{k_-}
	\end{bmatrix}\in{\mathbb C}^{k\times k}.
\end{equation}

Recently, Liu, So, and Wu~\cite{lisw:2019} laboriously analyzed how to solve
\begin{equation}\label{eq:Brockett}
\min_{X^{\HH}X=I_k}\trace(DX^{\HH}AX)
\end{equation}
by numerical optimization techniques, where $A\in\bbC^{n\times n}$ and $D\in\bbC^{k\times k}$ are Hermitian but $D$
may be indefinite. Its objective function $\trace(DX^{\HH}AX)$ in $X$ on the Stiefel manifold
$\{X\in\bbR^{n\times k}\,:\,X^{\HH}X=I_k\}$ is known as the {\em Brockett cost function} in the case when $D$
is diagonal and positive semi-definite, and optimizing it
has often been used as an illustrative example for optimization on the Stiefel manifold \cite[p.80]{abms:2008}, \cite{bicc:2019,broc:1991}.

Our goal in this paper is to go beyond the Brockett cost function to investigate,
as an extension of \eqref{eq:min-trace-Herm},
\begin{equation}\label{eq:min-trace-Herm2}
	\inf_{X^{\HH}BX=I_k\,\, (\mbox{\scriptsize or $-I_k$})}\trace(DX^{\HH}AX),
\end{equation}
where 
both $B$ and $D$ are possibly indefinite.
Our first main result is an  analytical solution
to \eqref{eq:min-trace-Herm2} for positive definite $B$ (for which only $X^{\HH}BX=I_k$ can be used as a constraint) in terms of the eigen-decompositions of $D$ and matrix pencil
$A-\lambda B$ and the solution lends itself
to be computed by more efficient numerical linear algebra techniques \cite{abbd:1999,bddrv:2000,demm:1997,li:2015}.
In particular, this result yields an elegant solution to the widely studied \eqref{eq:Brockett} \cite{abms:2008,lisw:2019}.
Our second main result is for a more general
setting that $B$ is indefinite and possibly singular and $A-\lambda B$ is a positive semi-definite matrix
pencil.
We show that the infimum in \eqref{eq:min-trace-Herm2} is finite if and only if $D$ positive semi-definite and
establish analytical solutions to it when the infimum is finite.

Note that whether $D$ is diagonal or not is inconsequential so long that it is Hermitian and the constraint is either
$X^{\HH}BX=I_k$ or $X^{\HH}BX=- I_k$ because we can always perform an eigen-decomposition $D=Q\Omega Q^{\HH}$ to get
$$
\min_{X^{\HH}BX=\pm I_k}\trace(DX^{\HH}AX)=\min_{\wtd X^{\HH}B\wtd X=\pm I_k}\trace(\Omega \wtd X^{\HH}A\wtd X),
$$
where $X$ and $\wtd X$ are related by $\wtd X=XQ$, $Q$ is unitary and $\Omega$ is diagonal of eigenvalues.
So far, we have been focusing on ``minimization'', but these formulations admit straightforward restatements for
``maximization'' by simply considering $-A$
instead.

The rest of this paper is organized as follows.
We state our main results for \eqref{eq:min-trace-Herm2} in section~\ref{sec:B-definite} for positive definite $B$ (and with $X^{\HH}BX=I_k$) and in section~\ref{sec:B-indefinite} for the more general setting that $B$ is genuinely indefinite.
The proofs for the main results are presented in sections~\ref{sec:thm:main-SPD:min:pf} and \ref{sec:main-SPD:min:sep:+},
respectively.
We draw our conclusion in section~\ref{sec:concl}.

\textbf{Notation.} Throughout this paper, ${\mathbb C}^{n\times m}$ is the set
of all $n\times m$ complex matrices, ${\mathbb C}^n={\mathbb C}^{n\times 1}$,
and ${\mathbb C}={\mathbb C}^1$. ${\mathbb R}$ is set of all real numbers.
$I_n$ (or simply $I$ if its dimension is
clear from the context) is the $n\times n$ identity matrix.
For a matrix $X\in\bbC^{m\times n}$, ${\cal N}(X)=\{x\in\bbC^n\,:\,Xx=0\}$ and $\cR(X)=\{Xx\,:\,x\in\bbC^n\}$
are the null space and the range of $X$, respectively.
$X^{\HH}$ is the conjugate transpose of a vector or matrix.
$A\succ 0$ ($A\succeq 0$) means that $A$ is Hermitian positive (semi-)definite, and $A\prec 0$ ($A\preceq 0$) if
$-A\succ 0$ ($-A\succeq 0$).

\section{Positive Definite $B$}\label{sec:B-definite}
Throughout this section and section~\ref{sec:thm:main-SPD:min:pf},
$A,\, B\in\bbC^{n\times n}$ and $D\in\bbC^{k\times k}$ are Hermitian and $B$ is positive definite. Then $A-\lambda B$ admits the following eigen-decomposition
\begin{equation}\label{eq:AB-eigD:SPD}
U^{\HH}AU=\Lambda\equiv\diag(\lambda_1,\lambda_2,\ldots,\lambda_n), \quad U^{\HH}BU=I_n,
\end{equation}
where $\lambda_i$ are the eigenvalues of $A-\lambda B$ and are, without loss of generality, arranged in the ascending order
as in \eqref{eq:lambda-ascending}, and $U=[u_1,u_2,\ldots,u_n]$ is the eigenvector matrix and $B$-unitary: $Au_i=\lambda_i Bu_i$
for all $i$ and $U^{\HH}BU=I_n$.
Let the eigen-decomposition of $D$ be
\begin{subequations}\label{eq:eigD4D}
\begin{equation}\label{eq:eigD4D-1}
Q^{\HH}DQ=\Omega\equiv\diag(\omega_1,\omega_2,\ldots,\omega_k),
\end{equation}
where $Q\in\bbC^{k\times k}$ is unitary and, without loss of generality,
\begin{equation}\label{eq:delta-odr}
	\omega_1\ge\cdots\ge\omega_{\ell}\ge 0\ge \omega_{\ell+1}\ge\cdots\ge\omega_k.
\end{equation}
\end{subequations}
The case $\ell=0$ or $\ell=k$ corresponds to when $D$ has no positive eigenvalues, i.e., $D\preceq 0$, or no nonnegative eigenvalues, i.e., $D\succeq 0$, respectively.

Our main result of this section is Theorem~\ref{thm:main-SPD:min} below.

\begin{theorem}\label{thm:main-SPD:min}
Suppose that $A,\, B\in\bbC^{n\times n}$ and $D\in\bbC^{k\times k}$ are Hermitian and $B$ is positive definite, admitting the
eigen-decompositions in \eqref{eq:AB-eigD:SPD} and \eqref{eq:eigD4D}. Then
\begin{equation}\label{eq:main-SPD:min}
\min_{X^{\HH}BX=I_k}\trace(DX^{\HH}AX)=\sum_{i=1}^{\ell}\omega_i\lambda_i+\sum_{i=\ell+1}^k\omega_i\lambda_{i+n-k}.
\end{equation}
Furthermore, any minimizer $X_{\opt}$ has the following characterizations:
\begin{enumerate}[{\rm (a)}]
\item If $D$ is nonsingular, then $\cR(X_{\opt}Q)$ is the eigenspace of $A-\lambda B$
      {\rm \cite[p.303]{stsu:1990}},
      associated with the $\ell$ smallest and $k-\ell$ largest eigenvalues  of $A-\lambda B$.
      If also all $\omega_i$ are distinct, then
      $$
      (X_{\opt}Q)^{\HH}AX_{\opt}Q=\diag(\underbrace{\lambda_1,\lambda_2,\ldots,\lambda_{\ell}}_{\ell},\underbrace{\lambda_{n-k+\ell+1},\ldots,\lambda_n}_{k-\ell}).
      $$
\item Suppose that $D$ is possibly singular and has $\ell_+$ positive eigenvalues and $\ell_-$ negative eigenvalues, and let $\what Q\in\bbC^{k\times (\ell_++\ell_-)}$ be the one from $Q$ by keeping its first $\ell_+$ and last $\ell_-$ columns. Then
    $\cR(X_{\opt}\what Q)$ is the eigenspace of $A-\lambda B$ associated with its $\ell_+$ smallest and $\ell_-$ largest eigenvalues. If also the nonzero eigenvalues of $D$ are distinct, then
    $$
    (X_{\opt}\what Q)^{\HH}AX_{\opt}\what Q
    =\diag(\underbrace{\lambda_1,\lambda_2,\ldots,\lambda_{\ell_+}}_{\ell_+},\underbrace{\lambda_{n-\ell_-+1},\ldots,\lambda_n}_{\ell_-}).
    $$
\end{enumerate}
\end{theorem}

There are a couple of remarks are in order. Firstly, the minimization extracts out the extreme eigenvalues
of $A-\lambda B$. Secondly, if all $\omega_i$ are distinct and nonzero, then the columns of $X_{\opt}Q$
are the associated eigenvectors. Thirdly, if $D$ does have $0$ as some of its eigenvalues, then in the notation of
Theorem~\ref{thm:main-SPD:min}(b), those eigenvalues $0$ can be matched to any $\lambda_i$ ($\ell_++1\le i\le n-\ell_-$), other things being equal, to still yield the same objective value as the optimal one in the right-hand side of \eqref{eq:main-SPD:min}.
Fourthly, upon replacing $A$ by $-A$, we obtain immediately the following corollary.

\begin{corollary}\label{cor:main-SPD:min}
Under the conditions of Theorem~\ref{thm:main-SPD:min}
\begin{equation}\label{eq:thm2}
\max_{X^{\HH}BX=I_k}\trace(DX^{\HH}AX)=\sum_{i=1}^{\ell}\omega_i\lambda_{n+1-i}+\sum_{i=\ell+1}^k\omega_i\lambda_{k-i+1}.
\end{equation}
Furthermore, any maximizer $X_{\opt}$ has the following characterizations:
\begin{enumerate}[{\rm (a)}]
\item If $D$ is nonsingular, then $\cR(X_{\opt}Q)$ is the eigenspace of $A-\lambda B$
      associated with the $k-\ell$ smallest and $\ell$ largest eigenvalues  of $A-\lambda B$.
      If also all $\omega_i$ are distinct, then
      $$
      (X_{\opt}Q)^{\HH}AX_{\opt}Q=
      \diag(\underbrace{\lambda_n,\lambda_{n-1},\ldots,\lambda_{n+1-\ell}}_{\ell},\underbrace{\lambda_{k-\ell},\ldots,\lambda_1}_{k-\ell}).
  $$
\item Suppose that $D$ is possibly singular and has $\ell_+$ positive eigenvalues and $\ell_-$ negative eigenvalues, and let $\what Q\in\bbC^{k\times (\ell_++\ell_-)}$ be the one from $Q$ by keeping its first $\ell_-$ and last $\ell_+$ columns. Then
    $\cR(X_{\opt}\what Q)$ is the eigenspace of $A-\lambda B$ associated with its $\ell_-$ smallest and $\ell_+$ largest eigenvalues. If also the nonzero eigenvalues of $D$ are distinct, then
    $$
    (X_{\opt}\what Q)^{\HH}AX_{\opt}\what Q=
    \diag(\underbrace{\lambda_n,\lambda_{n-1},\ldots,\lambda_{n+1-\ell_+}}_{\ell_+},\underbrace{\lambda_{\ell_-},\ldots,\lambda_1}_{\ell_-}).
    $$
\end{enumerate}
\end{corollary}

The proof of Theorem~\ref{thm:main-SPD:min} occupies a few pages and is deferred to section~\ref{sec:thm:main-SPD:min:pf}.

\section{Genuinely Indefinite $B$}\label{sec:B-indefinite}
Throughout this section and section~\ref{sec:main-SPD:min:sep:+}, $A-\lambda B\in\bbC^{n\times n}$ is a positive semi-definite matrix pencil, i.e.,
$A$ and $B$ are Hermitian and
there exists $\lambda_0\in \mathbb{R}$ such that $A-\lambda_0 B\succeq0$, and $B$ is genuinely indefinite in the sense that
$B$ has both positive and negative eigenvalues.
We are interested in  \eqref{eq:min-trace-Herm2}:
\begin{equation}\tag{\ref{eq:min-trace-Herm2}}
\inf\trace(DX^{\HH}AX)\quad\mbox{subject to either $X^{\HH}BX=I_k$ or $X^{\HH}BX=-I_k$}.
\end{equation}
After presenting our main results for it, we will discuss the
more generally constraint $X^{\HH}BX=J_k$, where $J_k$
is as given in \eqref{eq:Jk-dfn}.

Before we investigate \eqref{eq:min-trace-Herm2}, we review some of the related concepts and
results about a positive semi-definite matrix pencil $A-\lambda B$ \cite{lilb:2013}.
Let the integer triplet $(n_+, n_0, n_-)$
be the {\em inertia\/} of $B$, meaning $B$ has $n_+$ positive, $n_0$ 0, and $n_-$ negative eigenvalues,
respectively. Necessarily
\begin{equation}\label{eq:rankB}
	r:= \rank(B)=n_++n_-.
\end{equation}
We say $\mu\ne\infty$ is a {\em finite eigenvalue\/}
of $A-\lambda B$ if
\begin{equation}\label{eq:eig-dfn}
	\rank(A-\mu B)<\max_{\lambda\in{\mathbb C}}\rank(A-\lambda B),
\end{equation}
and $x\in{\mathbb C}^n$ is a corresponding {\em eigenvector\/} if $0\ne x\not\in{\cal N}(A)\cap{\cal N}(B)$ satisfies
\begin{equation}\label{eq:eigv-dfn}
	 Ax=\mu Bx,
\end{equation}
or equivalently, $0\ne x\in{\cal N}(A-\mu B)\backslash({\cal N}(A)\cap{\cal N}(B))$.
It is known \cite{lilb:2013} that
{\em a positive semi-definite pencil $A-\lambda B$ has only $r=\rank(B)$ finite eigenvalues all of which are real}.
Denote these finite eigenvalues by $\lambda_i^{\pm}$ ordered as
\begin{equation}\label{eq:finite-eig}
\lambda_{n_-}^-\le\cdots\le\lambda_1^-\le\lambda_1^+\le\cdots\le\lambda_{n_+}^+.
\end{equation}
It has been proved that for all $i,\,j$
\begin{equation}\label{eq:finite-eig-property}
	\lambda_i^-\le \lambda_0\le\lambda_j^+.
\end{equation}

As in section~\ref{sec:B-definite}, let $D$ have its eigen-decomposition  given by \eqref{eq:eigD4D}:
$$
Q^{\HH}DQ=\Omega\equiv\diag(\omega_1,\omega_2,\ldots,\omega_k), \quad
\omega_1\ge\omega_2\ge\cdots\ge\omega_k.
$$
Our main result of the section is Theorem~\ref{thm:main-SPD:min:sep:+} below.

\begin{theorem}\label{thm:main-SPD:min:sep:+}
Suppose that $A,\, B\in\bbC^{n\times n}$ and $D\in\bbC^{k\times k}$ are Hermitian, $A\ne 0$ and $B$ is
genuinely indefinite, $k\le n_+$, and the matrix pencil $A-\lambda B$ is positive semi-definite.
	Then
$$
\inf_{X^{\HH}BX=I_k}\trace(D X^{\HH}AX)>-\infty
$$
if and only if $D\succeq0$, in which case
\begin{equation}\label{eq:main-SPD:min:sep:+}
	\inf_{X^{\HH}BX=I_k}\trace(D X^{\HH}AX)=\sum_{i=1}^{k}\omega_i\lambda_i^+.
\end{equation}
The infimum can be attained, when $A-\lambda B$ is diagonalizable, by $X$ such that the columns of
$XQ$ are the eigenvectors of $A-\lambda B$ associated with its eigenvalues $\lambda_i^+$ for $1\le i\le k$, respectively.
\end{theorem}

Our proof of this theorem is rather involved and will be given in section~\ref{sec:main-SPD:min:sep:+}.
Apply Theorem~\ref{thm:main-SPD:min:sep:+} to the matrix pencil $A-(-\lambda)(-B)$, we immediately conclude the following corollary.

\begin{corollary}\label{cor:main-SPD:min:sep:-}
Suppose the conditions of Theorem~\ref{thm:main-SPD:min:sep:+}, except now $k\le n_-$. Then
$$
\inf_{X^{\HH}BX=-I_k}\trace(D X^{\HH}AX)>-\infty
$$
if and only if $D\succeq0$, in which case
\begin{equation}\label{eq:main-SPD:min:sep:-}
	\inf_{X^{\HH}BX=-I_k}\trace(D X^{\HH}AX)=-\sum_{i=1}^{k}\omega_i\lambda_i^-.
\end{equation}
The infimum can be attained, when $A-\lambda B$ is diagonalizable, by $X$ such that the columns of
$XQ$ are the eigenvectors of $A-\lambda B$ associated with its eigenvalues $\lambda_i^-$ for $1\le i\le k$, respectively.
\end{corollary}

Combining Theorem~\ref{thm:main-SPD:min:sep:+} and Corollary~\ref{cor:main-SPD:min:sep:-}, we present
a result for the more general constraint $X^{\HH}BX=J_k$, whose proof is deferred to section~\ref{sec:main-SPD:min:sep:+} as well.

\begin{corollary}\label{cor:main-SPD:min:sep}
Suppose that $A,\, B\in\bbC^{n\times n}$ and $D_{\pm}\in\bbC^{k_{\pm}\times k_{\pm}}$ are Hermitian, $A\ne 0$ and $B$ is
genuinely indefinite, $k_{\pm}\le n_{\pm}$, and the matrix pencil $A-\lambda B$ is positive semi-definite.
Let
$$
J_k=\begin{bmatrix}
		I_{k_+} & \\ & -I_{k_-}
	\end{bmatrix},\quad
D=\kbordermatrix{ &\sss k_+ &\sss k_- \\
	\sss k_+ & D_+ & \\
    \sss k_- &  & D_-},
$$
and denote by $\omega_1^+\ge\cdots\ge\omega_{k_+}^+$ and $\omega_1^-\ge\cdots\ge\omega_{k_-}^-$
the eigenvalues of $D_+$ and $D_-$, respectively.
If  $D_{\pm}\succeq0$, then
\begin{equation}\label{eq:main-SPD:min:sep}
	\inf_{X^{\HH}BX=J_k}\trace(DX^{\HH}AX)=\sum_{i=1}^{k_+}\omega_i^+\lambda_i^+-\sum_{i=1}^{k_-}\omega_i^-\lambda_i^-.
\end{equation}
The infimum can be attained when $A-\lambda B$ is diagonalizable.
\end{corollary}

One comment that we would like to emphasize about the conditions of Corollary~\ref{cor:main-SPD:min:sep} is that matrix $D$ has to take the same block-diagonal structure as $J_k$. Our next example shows that if $D$
doesn't have the same block-diagonal structure as $J_k$, the infimum may not be able to be expressed simply as some sum of
the products of the eigenvalues between $D$ and $A-\lambda B$.
The example involves a fair amount of complicated computation and can be skipped in the first reading.

Consider, given $\mu$ and $\delta$ such that $0<\delta<\frac{1}{\mu}<1<\mu$,
\begin{gather*}
A=\begin{bmatrix}
	1 & \\ & \mu\\
	\end{bmatrix},\,\,
B=\begin{bmatrix}
	1 & \\ & -1\\
	\end{bmatrix},\,\, J_2=B, \\
\Omega=\begin{bmatrix}
		1 & \\ & \delta
	\end{bmatrix}\succ 0, \,\,
Q=\begin{bmatrix}
	\sqrt{1-\sigma^2} & -\sigma\\
	\sigma & \sqrt{1-\sigma^2}
	\end{bmatrix},\,\,
D=Q^{\HH}\Omega Q\succ 0,
\end{gather*}
where  $ 0\ne\sigma\in(-1,1)$.
$A-\lambda B$ is positive definite pencil because $A-0\cdot B=A\succ 0$.
The two eigenvalues of $A-\lambda B$ are
$\lambda_1^-=-\mu$, $\lambda_1^+=1$.
$D$ doesn't have the same block structure as $J_2$. We will show that
\begin{equation}\label{eq:eg-arg-1}
\inf_{X^{\HH}BX=J_2}\trace(DX^{\HH}AX)<\min\{1+\delta\mu,\mu+\delta\}=1+\delta\mu,
\end{equation}
implying that the infimum cannot be simply expressed as any of the
two possible sums of products:$1+\delta\mu$ and $\mu+\delta$, between the eigenvalues
of $A-\lambda B$ and of $D$.

To see \eqref{eq:eg-arg-1}, we restrict $X$ to those $Y$:
$$
Y=\begin{bmatrix}
	\sqrt{1+\tau^2} & \tau\\
	\tau & \sqrt{1+\tau^2}
	\end{bmatrix}, \quad\tau\in[0,+\infty).
$$
It can be verified that $Y^{\HH}BY=J_2$.
In order to show \eqref{eq:eg-arg-1},
it suffices to show that $f(\sigma,\tau):=\trace(DY^{\HH}AY)<1+\delta\mu$ for some
$(\sigma,\tau)$.
We have
\begin{align*}
f(\sigma,\tau)&=\trace(DY^{\HH}AY)=\trace(Q^{\HH}\Omega QY^{\HH}AY) \\
	&=\trace\left(\begin{bmatrix}
		             1-\sigma^2+\delta\sigma^2 & (\delta-1)\sigma\sqrt{1-\sigma^2}\\
		             (\delta-1)\sigma\sqrt{1-\sigma^2} & \sigma^2+\delta(1-\sigma^2)\\
	               \end{bmatrix}
                   \begin{bmatrix}
	                 (1+\tau^2)+\mu\tau^2 & (1+\mu)\tau\sqrt{1+\tau^2}\\
	                 (1+\mu)\tau\sqrt{1+\tau^2} & \tau^2+\mu(1+\tau^2)\\
                 \end{bmatrix}\right)\\
   &=\big(1-\sigma^2+\delta\sigma^2\big)\cdot\big[(1+\tau^2)+\mu\tau^2\big]
        +2\cdot(\delta-1)\sigma\sqrt{1-\sigma^2}\cdot(1+\mu)\tau\sqrt{1+\tau^2} \\
   &\qquad +\big[ \sigma^2+\delta(1-\sigma^2)\big]\cdot\big[ \tau^2+\mu(1+\tau^2)\big] \\
   &=1+\delta\mu -\sigma^2(1-\delta)(1-\mu)+\tau^2(1+\delta)(1+\mu)
      +2(\delta-1)(1+\mu)\tau\sigma\sqrt{1-\sigma^2}\sqrt{1+\tau^2}\\
   &=1+\delta\mu+ (1+\delta)(1+\mu)\left[\tau^2-\gamma\nu\sigma^2 -2\gamma\tau\sigma\sqrt{1-\sigma^2}\sqrt{1+\tau^2}\right],
\end{align*}
where for the last equality, we set $\gamma:=\frac{1-\delta}{1+\delta}\in(\frac{\mu-1}{\mu+1},1)$ and
$\nu:=\frac{1-\mu}{1+\mu}\in(-1,0)$. In particular,
 $0<-\nu<\gamma<1$. We calculate the partial derivatives:
\begin{align*}
\frac{1}{2}\frac{\partial f(\sigma,\tau)}{\partial \tau}
 &=(1+\delta)(1+\mu)\left[\tau -\gamma\left(\sqrt{1+\tau^2}+\tau\frac{\tau}{\sqrt{1+\tau^2}}\right)\sigma\sqrt{1-\sigma^2}\right]\\
 &=(1+\delta)(1+\mu)\left[\tau -\gamma\frac{1+2\tau^2}{\sqrt{1+\tau^2}}\,\sigma\sqrt{1-\sigma^2}\right], 	\\
\frac{1}{2}\frac{\partial f(\sigma,\tau)}{\partial \sigma}
	&=(1+\delta)(1+\mu)\left[-\gamma\nu\sigma -\gamma\tau\sqrt{1+\tau^2}\left(\sqrt{1-\sigma^2}
                             +\sigma\frac{-\sigma}{\sqrt{1-\sigma^2}}\right)\right]	\\
    &=-(1+\delta)(1+\mu)\gamma\left[\nu\sigma +\tau\sqrt{1+\tau^2}\,\frac{1-2\sigma^2}{\sqrt{1-\sigma^2}}\right].
\end{align*}
The stationary points $(\sigma,\tau)$ of $f$ satisfy
\begin{equation}\label{eq:eg-arg-2}
\frac{1}{2}\frac{\partial f(\sigma,\tau)}{\partial \tau}=0,\quad
\frac{1}{2}\frac{\partial f(\sigma,\tau)}{\partial \sigma}=0.
\end{equation}
It can be seen that  $\tau=0 \iff \sigma=0$ from the system \eqref{eq:eg-arg-2}.
For $\tau\sigma\ne0$, \eqref{eq:eg-arg-2} yields
\[
	\frac{\tau\sqrt{1+\tau^2}}{1+2\tau^2}=\gamma\sigma\sqrt{1-\sigma^2},
	\qquad
	\tau\sqrt{1+\tau^2}=-\frac{\nu\sigma\sqrt{1-\sigma^2}}{1-2\sigma^2},
\]
or, equivalently
\begin{equation}\label{eq:eg-arg-2a}
	(1+2\tau^2)(1-2\sigma^2)=\frac{-\nu}{\gamma},
	\qquad
	\frac{\tau^2(1+\tau^2)}{1+2\tau^2}=-\gamma\nu\frac{\sigma^2(1-\sigma^2)}{1-2\sigma^2}.
\end{equation}
The first equation in \eqref{eq:eg-arg-2a} yields
\begin{equation}\label{eq:eg-arg-3}
\tau^2=-\frac{1}{2}\left[\frac{\nu}{\gamma(1-2\sigma^2)}+1\right],
\end{equation}
and plug it into the second equation in \eqref{eq:eg-arg-2a} to get
\begin{align}	
0&=\gamma\nu\sigma^2(1-\sigma^2)+\frac{-\gamma}{\nu}(1-2\sigma^2)^2\frac{-1}{4}
     \left[\frac{\nu}{\gamma(1-2\sigma^2)}+1\right]\left[1-\frac{\nu}{\gamma(1-2\sigma^2)}\right]
	\nonumber \\
&=\gamma\nu\sigma^2(1-\sigma^2)+\frac{\gamma}{4\nu}\left[(1-2\sigma^2)^2-\frac{\nu^2}{\gamma^2}\right]
	\nonumber \\
&=\frac{1}{4\gamma\nu}\left[\gamma^2-\nu^2-4\gamma^2(1-\nu^2)\sigma^2+4\gamma^2(1-\nu^2)\sigma^4\right],
	\nonumber \\
&=\frac{\gamma(1-\nu^2)}{\nu}\left[\frac{\gamma^2-\nu^2}{4\gamma^2(1-\nu^2)}-\sigma^2+\sigma^4\right].
   \label{eq:eg-arg-3a}
\end{align}
Since $(-1)^2-4\frac{\gamma^2-\nu^2}{4\gamma^2(1-\nu^2)}=\frac{\nu^2(1-\gamma^2)}{\gamma^2(1-\nu^2)}>0$,
solving \eqref{eq:eg-arg-3a} for $\sigma^2$ gives
$$
\sigma^2=\frac{1}{2}\left(1\pm\sqrt{1-\frac{\gamma^2-\nu^2}{\gamma^2(1-\nu^2)}}\right)
   =\frac{1}{2}\left(1\pm\sqrt{\frac{\nu^2(1-\gamma^2)}{\gamma^2(1-\nu^2)}}\right),
$$
which yields $1-2\sigma^2=\mp\left|\frac{\nu}{\gamma}\right|\sqrt{\frac{1-\gamma^2}{1-\nu^2}}$
and $1+2\tau^2=\pm\sign(\gamma\nu)\sqrt{\frac{1-\nu^2}{1-\gamma^2}}$ by \eqref{eq:eg-arg-3}.
Since $1+2\tau^2>0$, we conclude
$$
1+2\tau^2=\sqrt{\frac{1-\nu^2}{1-\gamma^2}}, \quad
1-2\sigma^2=-\frac{\nu}{\gamma}\sqrt{\frac{1-\gamma^2}{1-\nu^2}}.
$$
Hence stationary points $(\sigma_*,\tau_*)$ are determined by
\begin{equation}\label{eq:eg-arg-4'}
	\tau_*^2=\frac{1}{2}\left(\sqrt{\frac{1-\nu^2}{1-\gamma^2}}-1\right)>0,
	\qquad
	\sigma_*^2=\frac{1}{2}\left(\frac{\nu}{\gamma}\sqrt{\frac{1-\gamma^2}{1-\nu^2}}+1\right)\in(0,1),
\end{equation}
giving two stationary points $(\sigma_*,\tau_*)$:
\begin{equation}\label{eq:eg-arg-4}
	\tau_*=\sqrt{\frac{1}{2}\left(\sqrt{\frac{1-\nu^2}{1-\gamma^2}}-1\right)},
	\qquad
	\sigma_*^{\pm}=\pm\sqrt{\frac{1}{2}\left(\frac{\nu}{\gamma}\sqrt{\frac{1-\gamma^2}{1-\nu^2}}+1\right)}.
\end{equation}
The values of $f$ at  these stationary points are
\begin{align*}
	f(\sigma_*^{\pm},\tau_*)
  &=1+\delta\mu	+ \frac{1}{2}(1+\delta)(1+\mu)\times \\
  &\qquad \left[
		\nu\gamma-1+\sqrt{\frac{1-\nu^2}{1-\gamma^2}}-\nu^2\sqrt{\frac{1-\gamma^2}{1-\nu^2}}
	\mp\gamma\sqrt{\left(1-\frac{\nu^2}{\gamma^2}\frac{1-\gamma^2}{1-\nu^2}\right)
                   \left(1-\frac{1-\nu^2}{1-\gamma^2}\right)}\right]		\\
  &=1+\delta\mu	+ \frac{1}{2}(1+\delta)(1+\mu)\times \\
  &\qquad\left[
		\nu\gamma-1+\sqrt{\frac{1-\nu^2}{1-\gamma^2}}-\nu^2\sqrt{\frac{1-\gamma^2}{1-\nu^2}}
		\mp\sqrt{\frac{(\nu^2-\gamma^2)^2}{(1-\nu^2)(1-\gamma^2)}}\right]		\\
  &\hspace{3cm}\qquad(\text{let $\eta=\sqrt{\frac{1-\gamma^2}{1-\nu^2}}$})		\\
  &=1+\delta\mu	+ \frac{1}{2}(1+\delta)(1+\mu)\left[
		\nu\gamma-1+\frac{1}{\eta}-\nu^2\eta
		\mp\frac{\nu^2-\gamma^2}{(1-\nu^2)\eta}\right]		\\
  &=1+\delta\mu	+ \frac{1}{2}(1+\delta)(1+\mu)\left[
			\nu\gamma-1+(1-\nu^2)\eta+\frac{1-\eta^2}{\eta}
		\mp\frac{\eta^2-1}{\eta}\right]		\\
  &=1+\delta\mu	+ \frac{1}{2}(1+\delta)(1+\mu)\left[
			\nu\gamma-1+\sqrt{(1-\gamma^2)(1-\nu^2)}+\frac{1-\eta^2}{\eta}(1\pm1)
		\right]		\\
  &=1+\delta\mu	+ \frac{1}{2}(1+\delta)(1+\mu)\left[
			-2\frac{\delta+\mu-2\sqrt{\delta\mu}}{(1+\delta)(1+\mu)}+\frac{1-\eta^2}{\eta}(1\pm 1)
		\right]
		.
\end{align*}
It can be seen that  $f(\sigma_*^-,\tau_*)<f(0,0)$ because
$\delta+\mu-2\sqrt{\delta\mu}=(\sqrt{\delta}-\sqrt{\mu})^2>0$. This verifies \eqref{eq:eg-arg-1}.

\section{Proof of Theorem~\ref{thm:main-SPD:min}}\label{sec:thm:main-SPD:min:pf}
We start with three lemmas as preparation. The first lemma is about
a result from majorization \cite{bhat:1996,hojo:2013}.
Given two sets of real numbers $\{\alpha_i\}_{i=1}^m$ and
$\{\beta_i\}_{i=1}^m$, we say that $\{\beta_i\}_{i=1}^m$
{\em majorizes\/} $\{\alpha_i\}_{i=1}^m$ if
$$
\sum_{i=1}^j\alpha_i^{\downarrow}\le \sum_{i=1}^j\beta_i^{\downarrow}, \quad
\mbox{for $j=1,2,\ldots,m$}
$$
with equality holds for $j=m$,
where $\{\alpha_i^{\downarrow}\}_{i=1}^m$ is from re-ordering
$\{\alpha_i\}_{i=1}^m$ in the decreasing order, i.e.,
$$
\alpha_1^{\downarrow}\ge\alpha_2^{\downarrow}\ge\cdots\ge\alpha_m^{\downarrow}
$$
(similarly for $\{\beta_i^{\downarrow}\}_{i=1}^m$). We also use
notation $\alpha_i^{\uparrow}$ obtained from re-ordering $\{\alpha_i\}_{i=1}^m$
as well but in the increasing order.

\begin{lemma}\label{lm:YYZZ}
Let $\gamma_1\ge\gamma_2\ge\cdots\ge\gamma_m$. If
$\{\beta_i\}_{i=1}^m$
majorizes $\{\alpha_i\}_{i=1}^m$, then
\begin{equation}\label{eq:lm-ineq}
\sum_{i=1}^m\gamma_i\beta_i^{\uparrow}\le
\sum_{i=1}^m\gamma_i\alpha_i\le\sum_{i=1}^m\gamma_i\beta_i^{\downarrow}.
\end{equation}
Furthermore, if all $\gamma_i$ are distinct, then the first inequality becomes an equality if and only if $\alpha_i=\beta_i^{\uparrow}$ for all $i$.
Similarly, if all $\gamma_i$ are distinct, then the second inequality becomes an equality if and only if $\alpha_i=\beta_i^{\downarrow}$ for all $i$.
\end{lemma}

The first part of the lemma is exactly the same as \cite[Lemma 2.3]{li:2004c}, except that here it is not required
that all $\gamma_i\ge 0$.
 The second part on the inequalities becoming
equalities was not explicitly stated there, but it follows from the proof there straightforwardly. This lemma likely appeared elsewhere but an explicit reference is hard to find.
As a corollary, we have
$$
\sum_{i=1}^m\gamma_i\beta_i^{\uparrow}
\le\sum_{i=1}^m\gamma_i\beta_i
\le\sum_{i=1}^m\gamma_i\beta_i^{\downarrow}
$$
because clearly $\{\beta_i\}_{i=1}^m$ majorizes $\{\beta_i\}_{i=1}^m$ itself.

\begin{proof}[Proof of Lemma~\ref{lm:YYZZ}]
Without loss of generality, we may assume $\gamma_m> 0$; otherwise
we can always pick a scalar $\xi$ such that $\gamma_m+\xi\ge 0$, and let
$$
\tilde\gamma_i:=\gamma_i+\xi> 0\quad\mbox{for $1\le i\le m$}.
$$
By assumption, we have $\sum_{i=1}^m\alpha_i=\sum_{i=1}^m\beta_i=\sum_{i=1}^m\beta_i^{\uparrow}=\sum_{i=1}^m\beta_i^{\downarrow}=:\eta$, and thus
$$
\sum_{i=1}^m\gamma_i\beta_i^{\uparrow}=-\xi\eta+\sum_{i=1}^m\tilde\gamma_i\beta_i^{\uparrow},\,\,
\sum_{i=1}^m\gamma_i\alpha_i=-\xi\eta+\sum_{i=1}^m\tilde\gamma_i\alpha_i,\,\,
\sum_{i=1}^m\gamma_i\beta_i^{\downarrow}=-\xi\eta+\sum_{i=1}^m\tilde\gamma_i\beta_i^{\downarrow}.
$$
It suffices to prove the lemma for $\tilde\gamma_1\ge\tilde\gamma_2\ge\cdots\ge\tilde\gamma_m> 0$, instead.

The argument below up to \eqref{ineq:2} appears in the proof of \cite[Lemma 2.3]{li:2004c}. It is repeated here for
the purpose of arguing when the equality signs in \eqref{eq:lm-ineq} are attained.
Suppose $\gamma_1\ge\gamma_2\ge\cdots\ge\gamma_m> 0$ and set
$$
p_j=\sum_{i=1}^j\beta_i^{\uparrow},\quad
s_j=\sum_{i=1}^j\alpha_i, \quad t_j=\sum_{i=1}^j\beta_i^{\downarrow},\quad
p_0=s_0=t_0=0.
$$
Since $\{\beta_i\}_{i=1}^m$ majorizes $\{\alpha_i\}_{i=1}^m$, we have
$$
p_j\le s_j\le t_j,\quad p_m=s_m=t_m
$$
and thus
\begin{eqnarray}
\sum_{i=1}^m\gamma_i\alpha_i&=&\sum_{i=1}^m(s_i-s_{i-1})\gamma_i \nonumber\\
  &=&\sum_{i=1}^ms_i\gamma_i-\sum_{i=2}^ms_{i-1}\gamma_i \nonumber\\
  &=&s_m\gamma_m+\sum_{i=1}^{m-1}s_i(\gamma_i-\gamma_{i+1})\nonumber\\
  &\le&t_m\gamma_m+\sum_{i=1}^{m-1}t_i(\gamma_i-\gamma_{i+1})\nonumber\\
  &=&\sum_{i=1}^m\gamma_i\beta_i^{\downarrow}, \label{ineq:1}\\
\sum_{i=1}^m\gamma_i\alpha_i
      &=&s_m\gamma_m+\sum_{i=1}^{m-1}s_i(\gamma_i-\gamma_{i+1}) \nonumber\\
      &\ge&p_m\gamma_m+\sum_{i=1}^{m-1}p_i(\gamma_i-\gamma_{i+1})\nonumber\\
      &=&\sum_{i=1}^m\gamma_i\beta_i^{\uparrow}, \label{ineq:2}
\end{eqnarray}
as required. To figure out when any of the inequality in the lemma is an equality, we look at \eqref{ineq:1}, for an example. We notice that there is only one inequality sign during the derivation in \eqref{ineq:1}. In order for the inequality to become an equality,
assuming all $\gamma_i$ are distinct, we will have to have $s_i=t_i$ for all $i$ and consequently, $\alpha_i=\beta_i^{\downarrow}$ for all $i$.
\end{proof}

The next two lemmas relate the diagonal entries of a Hermitian matrix with its eigenvalues.

\begin{lemma}[{\cite[Exercise II.1.12, p.35]{bhat:1996}}]\label{lm:diag-by-eig}
The multiset of the diagonal entries of
a Hermitian matrix is majorized by the multiset of its eigenvalues.
\end{lemma}

\begin{lemma}\label{lm:diag=eig}
For a Hermitian matrix, if the multiset of its diagonal entries is the same as the multiset of its  eigenvalues, then it is diagonal.
\end{lemma}

\begin{proof}
This lemma is probably known, but we could not find a reference to it. For completeness, we provide a quick proof.
Let $A=[a_{ij}]\in\bbC^{n\times n}$ be such a Hermitian matrix with eigenvalues $\{\lambda_i\}_{i=1}^n$. By the assumption,
$$
\|A\|_{\F}^2=\sum_{i,j=1}^n|a_{ij}|^2=\sum_{i=1}^n|\lambda_i|^2=\sum_{i=1}^n|a_{ii}|^2,
$$
where $\|A\|_{\F}$ denotes the Frobenius norm of $A$. Hence $|a_{ij}|^2=0$ for all $i\ne j$, as expected.
\end{proof}


Now we are ready to prove Theorem~\ref{thm:main-SPD:min}.


\begin{proof}[Proof of Theorem~\ref{thm:main-SPD:min}]
Recall the  eigen-decomposition \eqref{eq:AB-eigD:SPD} with \eqref{eq:lambda-ascending} for $A-\lambda B$
and the eigen-decomposition \eqref{eq:eigD4D} for $D$. Consider first that $D$ is nonsingular, i.e.,
all $\omega_i\ne 0$.

Introducing
\begin{equation}\label{eq:main-SPD:min:pf-1}
Y=U^{-1}XQ
\quad\Rightarrow\quad
X=UY Q^{\HH},
\end{equation}
we find that the left-hand side of \eqref{eq:main-SPD:min} can be transformed to
$$
\min_{X^{\HH}BX=I_k}\trace(DX^{\HH}AX)=\min_{Y^{\HH}Y=I_k}\trace(\Omega Y^{\HH}\Lambda Y),
$$
and any minimizer of one yield a  minimizer of the other according to \eqref{eq:main-SPD:min:pf-1}.

For any given $Y\in\bbC^{n\times k}$ with $Y^{\HH}Y=I_k$, denote the eigenvalues of $Y^{\HH}\Lambda Y$ by
$$
\mu_1\le\mu_2\le\cdots\le\mu_k,
$$
where we suppress the dependency of  $\mu_i$ on $Y$ for clarity.
Cauchy's interlacing inequalities say that
\begin{equation}\label{eq:Cauchy}
\mbox{$\lambda_{i+n-k}\ge\mu_i\ge\lambda_i$ for all $1\le i\le k$}.
\end{equation}
Denote the diagonal entries of $Y^{\HH}\Lambda Y$ by $(Y^{\HH}\Lambda Y)_{ii}$ for $i=1,2,\ldots, k$, and let
$\alpha_i$ be the reordering of $(Y^{\HH}\Lambda Y)_{ii}$ in the increasing order, i.e.,
$$
\alpha_1\le\alpha_2\le\cdots\le\alpha_k.
$$
Evidently,
$\{(Y^{\HH}\Lambda Y)_{ii}\}_{i=1}^k$ is majorized by $\{\alpha_i\}_{i=1}^k$ because they are the same up to a permutation.
By Lemma~\ref{lm:diag-by-eig}, $\{\alpha_i\}_{i=1}^k$ is majorized by $\{\mu_i\}_{i=1}^k$.
We have
\begin{align}
\trace(\Omega Y^{\HH}\Lambda Y)&=\sum_{i=1}^k\omega_i(Y^{\HH}\Lambda Y)_{ii} \quad\mbox{(use Lemma~\ref{lm:YYZZ})}\nonumber \\
      &\ge\sum_{i=1}^k\omega_i\alpha_i \quad\mbox{(use Lemma~\ref{lm:YYZZ})}\nonumber\\
      &\ge\sum_{i=1}^k\omega_i\mu_i \nonumber\\
      &=\sum_{i=1}^{\ell}\omega_i\mu_i+\sum_{i=\ell+1}^k\omega_i\mu_i \quad\mbox{(use \eqref{eq:Cauchy})} \nonumber\\
      &\ge\sum_{i=1}^{\ell}\omega_i\lambda_i+\sum_{i=\ell+1}^k\omega_i\lambda_{i+n-k}, \label{eq:thm1:pf-2}
\end{align}
Since $Y$ is arbitrary, we have
\begin{equation}\label{eq:tmm1-1}
\min_{Y^{\HH}Y=I_k}\trace(\Omega Y^{\HH}\Lambda Y)\ge\sum_{i=1}^{\ell}\omega_i\lambda_i+\sum_{i=\ell+1}^k\omega_i\lambda_{i+n-k}.
\end{equation}
It is not too hard to find a particular $Y$ such that
$\trace(\Omega Y^{\HH}\Lambda Y)$ is equal to the right-hand side of \eqref{eq:tmm1-1}.
Therefore we have \eqref{eq:main-SPD:min}.

Suppose now all $\omega_i$ are distinct and $Y_{\opt}$ is a minimizer. We must have
\begin{align}
\trace(\Omega Y_{\opt}^{\HH}\Lambda Y_{\opt})&=\sum_{i=1}^k\omega_i(Y_{\opt}^{\HH}\Lambda Y_{\opt})_{ii} \nonumber \\
       &=\sum_{i=1}^k\omega_i\alpha_i \label{eq:tmm1-2a}\\
       &=\sum_{i=1}^k\omega_i\mu_i \label{eq:tmm1-2b}\\
       &=\sum_{i=1}^{\ell}\omega_i\lambda_i+\sum_{i=\ell+1}^k\omega_i\lambda_{i+n-k}, \label{eq:tmm1-2c}
\end{align}
where $\alpha_1\le\alpha_2\le\cdots\le\alpha_k$ are the reordering of $(Y_{\opt}^{\HH}\Lambda Y_{\opt})_{ii}$,
and $\mu_1\le\mu_2\le\cdots\le\mu_k$ are the eigenvalues of $Y_{\opt}^{\HH}\Lambda Y_{\opt}$.
For the equalities in \eqref{eq:tmm1-2a} -- \eqref{eq:tmm1-2c} to hold, we must have for all $i$
\begin{align*}
&(Y_{\opt}^{\HH}\Lambda Y_{\opt})_{ii}=\alpha_i=\mu_i=\lambda_i\quad\mbox{for $1\le i\le \ell$}, \\
&(Y_{\opt}^{\HH}\Lambda Y_{\opt})_{ii}=\alpha_i=\mu_i=\lambda_{n-k+i}\quad\mbox{for $\ell+1\le i\le k$},
\end{align*}
and
$
Y_{\opt}^{\HH}\Lambda Y_{\opt}=\diag(\alpha_1,\alpha_2,\ldots,\alpha_k)
    =\diag(\lambda_1,\lambda_2,\ldots,\lambda_{\ell},\lambda_{n-k+\ell+1},\ldots,\lambda_n).
$
Now use the relation \eqref{eq:main-SPD:min:pf-1} to conclude the proof for the case when $D$ is nonsingular.

Return to the case when $D$ is singular, i.e., some of its eigenvalues $\omega_i=0$. Let $\what Q$ be as the one
defined in item (b) and $\what Q_{\bot}$ be the columns of $Q$ not in $\what Q$.
The eigen-decomposition \eqref{eq:eigD4D} of $D$ can be rewritten as
$$
D=[\what Q,\what Q_{\bot}]\begin{bmatrix}
                          \what\Omega & 0 \\
                          0 & 0
                        \end{bmatrix}[\what Q,\what Q_{\bot}]^{\HH},
$$
where $\what\Omega=\diag(\omega_1,\ldots,\omega_{\ell_+},\omega_{k-\ell_-+1},\ldots,\omega_k)$ of all nonzero eigenvalues of $D$.
It can be verified that
\begin{equation}\label{eq:thm1:pf-3}
\trace(DX^{\HH}AX)=\trace(\what\Omega\what Y^{\HH}A\what Y),
\end{equation}
where $\what Y=X\what Q$, given $X$. If $X^{\HH}BX=I_k$, then $Y^{\HH}BY=\what Q^{\HH}X^{\HH}BX\what Q=\what Q^{\HH}\what Q=I_{\ell_++\ell_-}$.
On the other hand, given $\what Y\in\bbC^{n\times (\ell_++\ell_-)}$ such that $Y^{\HH}BY=I_{\ell_++\ell_-}$, we can expand
it to $Y=[\what Y,\what Y_c]\in\bbC^{n\times k}$ such that $Y^{\HH}BY=I_k$ and then let $X=Y[\what Q,\what Q_{\bot}]^{\HH}$ for which it can be seen that \eqref{eq:thm1:pf-3} holds. This proves
\begin{equation}\label{eq:thm1:pf-4}
\min_{X^{\HH}BX=I_k}\trace(DX^{\HH}AX)=\min_{\what Y^{\HH}B\what Y=I_{\ell_++\ell_-}}\trace(\what\Omega\what Y^{\HH}A\what Y),
\end{equation}
and a maximizer for one leads to the maximizer for the other. The right-hand side of \eqref{eq:thm1:pf-4} is a minimization
problem belonging to
the case of nonsingular $D$ that we just dealt with.
\end{proof}

\section{Proof of Theorem~\ref{thm:main-SPD:min:sep:+}}\label{sec:main-SPD:min:sep:+}
In preparing for the proof of Theorem~\ref{thm:main-SPD:min:sep:+},
we may assume, without loss of generality, that $A\succeq0$;
Otherwise, noticing
\begin{equation}\label{eq:simplify0}
\trace(DX^{\HH}AX)
	=\trace(DX^{\HH}(A-\lambda_0B)X)+\lambda_0\trace(D),
\end{equation}
we may consider $\trace(DX^{\HH}(A-\lambda_0B)X)$, instead.

In what follows, suppose that  $A\succeq 0$.

By \cite[Lemma~3.8]{lilb:2013}, $A-\lambda B$ admits an eigen-decomposition as follows. There exists a nonsingular matrix $U\in \bbC^{n\times n}$ such that
\begin{subequations}\label{eq:eigen-decomp}
\begin{align}
U^{\HH}AU & = \kbordermatrix{ &\sss n_+-m_0 &\sss n_--m_0 & \sss 2m_0 &n_0 \\
               \sss n_+-m_0 & \Lambda_+ & & & \\
               \sss n_--m_0 & & -\Lambda_- & & \\
               \sss 2m_0 & & & \Lambda_{\rm b} & \\
               \sss n_0 & & & & \Lambda_{\infty}}
             =:\kbordermatrix{ &\sss r & \sss n_0 \\
                  \sss r & \Lambda_r & \\
                  \sss n_0 & & \Lambda_{\infty}}=:\Lambda\succeq0, \label{eq:eigen-decomp-1} \\
U^{\HH}BU & = \kbordermatrix{ &\sss n_+-m_0 &\sss n_--m_0 & \sss 2m_0 &n_0 \\
               \sss n_+-m_0 & I_{n_+-m_0} & & & \\
               \sss n_--m_0 & & -I_{n_--m_0} & & \\
               \sss 2m_0 & & & J_{\rm b} & \\
               \sss n_0  & & & & 0   }
               =:\kbordermatrix{ &\sss r & \sss n_0 \\
                              \sss r & J_r & \\
                  \sss n_0 & & J_{\infty}}=:J_n, \label{eq:eigen-decomp-2}
\end{align}
where $0\le m_0\le\min\{n_+,n_-\}$,
and\footnote {Recall the simplification due to \eqref{eq:simplify0}. In general,
          $\Lambda_0$ in \eqref{eq:eigen-decomp-5} takes the form $\begin{bmatrix}
	                                                          0&\lambda_0\\ \lambda_0&1\\
                                                              \end{bmatrix}$, and thus
          $\lambda_{m_0}^-=\cdots=\lambda_1^-=\lambda_0=\lambda_1^+=\cdots=\lambda_{m_0}^+$.}
\begin{gather}
U=\kbordermatrix{ &\sss n_+-m_0 &\sss n_--m_0 & \sss 2m_0 &n_0 \\
                  & U_+ & U_-       & U_{\rm b}     & U_{\infty} }
                  =:\kbordermatrix{ &\sss r & \sss n_0 \\
                              & U_r&U_{\infty}}, \label{eq:eigen-decomp-3} \\
\Lambda_+=\diag(\lambda_{m_0+1}^+,\dots,\lambda_{n_+}^+),\quad
\Lambda_-=\diag(\lambda_{n_-}^-,\dots,\lambda_{m_0+1}^-), \label{eq:eigen-decomp-4}\\
\Lambda_0=\begin{bmatrix} 0&0\\ 0&1 \end{bmatrix},\quad
             F_2=\begin{bmatrix}
                     0&1\\ 1&0
                 \end{bmatrix},  \label{eq:eigen-decomp-5}\\
J_{\rm b}=\diag(\underbrace{F_2,\dots,F_2}_{m_0}),\quad
\Lambda_{\rm b}=\diag(\underbrace{\Lambda_0,\dots,\Lambda_0}_{m_0}),\quad
\Lambda_{\infty}\succeq 0. \label{eq:eigen-decomp-6}
\end{gather}
\end{subequations}
Both $\Lambda$ in \eqref{eq:eigen-decomp-1} and $J_n$ in \eqref{eq:eigen-decomp-2} are diagonal if $m_0=0$, i.e.,
in the absence of blocks $\Lambda_{\rm b}$, $J_{\rm b}$, and $U_{\rm b}$. For the case, we say that $A-\lambda B$
is {\em diagonalizable}.
It can be seen from \eqref{eq:eigen-decomp} that the finite eigenvalues of $A-\lambda B$ are given by
\begin{align*}
\lambda_{n_-}^-\le\cdots\le\lambda_{m_0+1}^-\le\underbrace{ 0 =\cdots= 0 }_{m_0}
		=\underbrace{ 0 =\cdots= 0 }_{m_0}
         \le\lambda_{m_0+1}^+\le\cdots\le\lambda_{n_+}^+,
\end{align*}
which, compared to \eqref{eq:finite-eig-property}, implies
$\lambda_{m_0}^-=\cdots=\lambda_1^-=0=\lambda_1^+=\cdots=\lambda_{m_0}^+$, and they come from $\Lambda_{\rm b}-\lambda J_{\rm b}$.

Letting $Y=U^{-1}XQ$, we can transform \eqref{eq:min-trace-Herm2} for the case $X^{\HH}BX=I_k$ into
\begin{equation}\label{eq:min-trace-Herm2':+}
\inf_{X^{\HH}BX=I_k}\trace(DX^{\HH}AX)
	=\inf_{Y^{\HH}J_n Y =I_k}\trace(\Omega Y^{\HH}\Lambda Y),
\end{equation}
where $k\le n_+$.  The next two lemmas will be needed in our later proof.

\begin{lemma}[{\cite[Corollary~5.12]{vese:2011}}]\label{lm:complement-basis:+}
Let $J_n=\diag(I_{n_+},-I_{n_-})$ and $n=n_++n_-$. Any vector set $u_1,\dots,u_k$ satisfying
$u_i^{\HH}J_nu_j=\pm \delta_{ij}$ for $i,j=1,\dots,k$ can be complemented to a basis $\{u_1,\dots,u_n\}$ of $\bbC^n$ satisfying $u_i^{\HH}J_nu_j=\pm\delta_{ij}$ for $i,j=1,\dots,n$, where $\delta_{ij}$ is the Kronecker delta which is $1$ for $i=j$ and $0$ otherwise, and the numbers of $1$ and $-1$ among $u_i^{\HH}J_nu_i$ for $1\le i\le n$ are $n_+$ and $n_-$, respectively.
\end{lemma}

\begin{lemma}[{\cite[Example~6.3]{vese:2011}}]\label{lm:polar-decomposition:+}
Let $J_n=\diag(I_{n_+},-I_{n_-})$. A matrix $X$ satisfies $X^{\HH}J_nX=J_n$ if and only if it is of the form
\begin{equation}\label{eq:lm:polar-decomposition:+}
	X=\begin{bmatrix}
		(I_{n_+}+ W  W ^{\HH})^{1/2} &  W \\
		 W ^{\HH} & (I_{n_-}+ W ^{\HH} W )^{1/2}\\
		\end{bmatrix}\begin{bmatrix}
		 V_+ & \\ & V_-
	\end{bmatrix},
\end{equation}
where $V_+\in \bbC^{n_+\times n_+}$ and $V_-\in \bbC^{n_-\times n_-}$ are unitary, and $W\in \bbC^{n_+\times n_-}$.
\end{lemma}

Lemma~\ref{lm:complement-basis:+} can also be found in many classical monographs, e.g., \cite{malc:1963,golr:2005}, and
Lemma~\ref{lm:polar-decomposition:+} can be found in \cite{vese:1993,kove:1995}, where \eqref{eq:lm:polar-decomposition:+} is called a (hyperbolic) {\em polar decomposition\/} of $X$. Now we are ready to present our proof.

\begin{proof}[Proof of Theorem~\ref{thm:main-SPD:min:sep:+}]
First we deal with the case when \textbf{the matrix $B$ is singular}, i.e., $n_0>0$ in \eqref{eq:eigen-decomp}.
Partition $Y=\kbordermatrix{ &\sss k\\
	\sss r & Y_r \\
    \sss n_0 & Y_{\infty} }$,
and then
\begin{align}
\inf_{Y^{\HH}J_n Y =I_k}\trace(\Omega  Y^{\HH}\Lambda Y)
	&= \inf_{Y_r^{\HH}J_r Y_r =I_k}\left[ \trace(\Omega  Y_r^{\HH}\Lambda_r Y_r)
                         +\trace(\Omega  Y_{\infty}^{\HH}\Lambda_{\infty} Y_{\infty}) \right] \nonumber	\\
	&= \inf_{Y_r^{\HH}J_r Y_r =I_k}\trace(\Omega  Y_r^{\HH}\Lambda_r Y_r)
        +\inf_{Y_{\infty}}\trace(\Omega  Y_{\infty}^{\HH}\Lambda_{\infty} Y_{\infty}).\label{eq:B=singular:+}
\end{align}
We will examine the two terms in \eqref{eq:B=singular:+} separately.
Constraint $Y^{\HH}J_nY=I_k$ yields $Y_r^{\HH}J_rY_r=I_k$, leaving $Y_{\infty}\in\bbC^{n_0\times k}$ arbitrary.
Restricting $Y_{\infty}$ to a rank-1 matrix $xy^{\HH}$, we find
\begin{align*}
\inf_{Y_{\infty}}\trace(\Omega  Y_{\infty}^{\HH}\Lambda_{\infty} Y_{\infty})
	&\le \inf_{\rank(Y_{\infty})\le 1}\trace(\Omega  Y_{\infty}^{\HH}\Lambda_{\infty} Y_{\infty}) 	\\
    &= \inf_{y,x}\trace(\Omega  x y^{\HH}\Lambda_{\infty} y x^{\HH}) 	\\
	&= \inf_{y,x} (x^{\HH}\Omega  x) (y^{\HH}\Lambda_{\infty} y).
\end{align*}
There are three cases.
\begin{enumerate}
\item $\Omega \not\succeq 0$ and $\Lambda_{\infty}\ne0$: we have
		$\inf_{y,x} (x^{\HH}\Omega  x) (y^{\HH}\Lambda_{\infty} y)=-\infty$, which leads to that the second
      infimum in \eqref{eq:B=singular:+} is $-\infty$.
\item $\Omega \not\succeq 0$ and $\Lambda_{\infty}=0$: we have
		$Y_{\infty}^{\HH}\Lambda_{\infty} Y_{\infty}=0$, which leads to that the second
      infimum in \eqref{eq:B=singular:+} is $0$. But our later proof for
      nonsingular $B$ shows that for the case the first infimum in \eqref{eq:B=singular:+} is $-\infty$.
\item $\Omega \succeq 0$: we have $Y_{\infty}^{\HH}\Lambda_{\infty} Y_{\infty}\succeq0$, and
      $\trace(\Omega  Y_{\infty}^{\HH}\Lambda_{\infty} Y_{\infty})\ge0$ and
          $\trace(\Omega  Y_{\infty}^{\HH}\Lambda_{\infty} Y_{\infty})=0$ for $Y_{\infty}=0$, which leads to that
      the second       infimum in \eqref{eq:B=singular:+} is $0$.
\end{enumerate}
The first infimum in \eqref{eq:B=singular:+}: $\inf_{Y_r^{\HH}J_r Y_r =I_k}\trace(\Omega  Y_r^{\HH}\Lambda_r Y_r)$,
falls into the case when \textbf{the matrix $B$ is nonsingular}, which we are about to investigate.

Suppose now that $B$ is nonsingular, i.e., $n_0=0$ in \eqref{eq:eigen-decomp}.

Consider first  that $m_0=0$, namely \textbf{the pencil $A-\lambda B$ is also diagonalizable}.
Then $J_n=\diag(I_{n_+},-I_{n_-})$.
Since $Y^{\HH}J_nY=I_k$,
by Lemma~\ref{lm:complement-basis:+} we can complement $Y$ to $\wtd Y=\begin{bmatrix}
	Y & Y_c
\end{bmatrix}\in\bbC^{n\times n}$ such that
$\wtd Y^{\HH}J_n\wtd Y=J_n $.
By 
Lemma~\ref{lm:polar-decomposition:+}, $\wtd Y$ has a hyperbolic polar decomposition
\begin{equation}\label{eq:main-SPD:min:sep:+:pf-1}
	\wtd Y=\begin{bmatrix}
		(I_{n_+}+\wtd\Sigma \wtd\Sigma ^{\HH})^{1/2} & \wtd\Sigma \\
		\wtd\Sigma ^{\HH} & (I_{n_-}+\wtd\Sigma ^{\HH}\wtd\Sigma )^{1/2}\\
		\end{bmatrix}\begin{bmatrix}
		\wtd V_+ & \\ &\wtd V_-
	\end{bmatrix},
\end{equation}
where $\wtd V_+\in \bbC^{n_+\times n_+}$ and $\wtd V_-\in \bbC^{n_-\times n_-}$ are unitary,
and $\wtd\Sigma\in\bbC^{n_+\times n_-}$.
Let $\wtd\Sigma =W_+\Sigma W_-^{\HH}$ be the singular value decomposition of $\wtd\Sigma$,
where
$$
\Sigma = \begin{bmatrix}
	\Sigma_0 \\0\\
	\end{bmatrix}\,\,\mbox{if $n_+\ge n_-$,\quad or}\,\,
\Sigma = \begin{bmatrix}
	0 &\Sigma_0
\end{bmatrix}\,\,\mbox{if $n_+< n_-$}.
$$
Hence plug $\wtd\Sigma =W_+\Sigma W_-^{\HH}$ into \eqref{eq:main-SPD:min:sep:+:pf-1} to turn
$\wtd Y=WS V^{\HH}$, where
\begin{subequations}\label{eq:main-SPD:min:sep:+:pf-2}
\begin{equation}\label{eq:main-SPD:min:sep:+:pf-2a}
W=\kbordermatrix{ &\sss n_+ &\sss n_- \\
		\sss n_+ & W_+ & \\
        \sss n_- &     & W_-		},\quad
V=\kbordermatrix{ &\sss n_+ &\sss n_- \\
		\sss n_+ & V_+ & \\
        \sss n_- &      &V_-}
		:=\kbordermatrix{ &\sss n_+ &\sss n_- \\
		\sss n_+ & \wtd V_+^{\HH}W_+ & \\
        \sss n_- & &\wtd V_-^{\HH}W_-},
\end{equation}
and
\begin{align}
S&=\begin{bmatrix}
		(I_{n_+}+\Sigma\Sigma^{\HH})^{1/2} & \Sigma\\
		\Sigma^{\HH} & (I_{n_-}+\Sigma^{\HH}\Sigma)^{1/2}\\
	\end{bmatrix} \nonumber \\
 &=\begin{bmatrix}
		(I+\Sigma_0^2)^{1/2} & 0 & \Sigma_0\\
		0 & I_{|n_+-n_-|} & 0 \\
		\Sigma_0 & 0 &(I+\Sigma_0^2)^{1/2}\\
	\end{bmatrix}	. \label{eq:main-SPD:min:sep:+:pf-2b}
\end{align}
\end{subequations}
Noticing $Y=\wtd YI_{n;k}$ where $I_{n;k}=\begin{bmatrix}
	I_k\\ 0
\end{bmatrix}\in\bbC^{n\times k}$, we have from \eqref{eq:min-trace-Herm2':+}
\begin{align}
\inf_{Y^{\HH}J_n Y =I_k}\trace(\Omega  Y^{\HH}\Lambda Y)
	&= \inf_{\wtd Y^{\HH}J_n \wtd Y=J_n }\trace(\Omega  I_{n;k}^{\HH}\wtd Y^{\HH}\Lambda \wtd YI_{n;k}) \nonumber \\
	&= \inf_{\Sigma_0\succeq0~\text{diagonal}\atop V_+,V_-,W_+,W_-~\text{unitary}}\trace(I_{n;k}\Omega I_{n;k}^{\HH}VSW^{\HH}\Lambda WSV^{\HH}) \nonumber \\
    &=\inf_{\Sigma_0\succeq0~\text{diagonal}\atop V_+,V_-,W_+,W_-~\text{unitary}}\trace(\wtd \Omega_VS\Lambda_W S), \label{eq:min-trace-Herm2'':+}
\end{align}
where $\wtd\Omega_V =V^{\HH}I_{n;k}\Omega I_{n;k}^{\HH}V$ and $\Lambda_W=W^{\HH}\Lambda W$.
Use \eqref{eq:main-SPD:min:sep:+:pf-2a} to see
$$
\wtd\Omega_V=\kbordermatrix{ & \sss n_+ & \sss n_- \\
		\sss n_+ & \wtd\Omega_{+,V}  & \\
	\sss n_- &  & 0}, \,\,
\Lambda_W=
\kbordermatrix{ & \sss n_+ & \sss n_- \\
	\sss n_+ & \Lambda_{+,W} & \\
	\sss n_- & & -\Lambda_{-,W}},
$$
where
$
\wtd \Omega_+ =\kbordermatrix{ & \sss k & \sss n_+-k \\
	\sss k & \Omega  & \\
	\sss n_+-k &  & 0}$,
$\wtd\Omega_{+,V}=V_+^{\HH}\wtd \Omega_+V_+$, $\Lambda_{+,W}=W_+^{\HH}\Lambda_+ W_+$, and $\Lambda_{-,W}=W_-^{\HH}\Lambda_- W_-$.
As a result, $\trace(\wtd \Omega_VS\Lambda_W S)$ can be given by
\begin{align*}
&\trace(\begin{bmatrix}
		\wtd \Omega_{+,V} & \\
		& 0 \\
	\end{bmatrix}\begin{bsmallmatrix}
		(I_{n_+}+\Sigma\Sigma^{\HH})^{1/2} & \Sigma\\
		\Sigma^{\HH} & (I_{n_-}+\Sigma^{\HH}\Sigma)^{1/2}\\
	\end{bsmallmatrix}\begin{bmatrix}
	\Lambda_{+,W} & \\
	& \Lambda_{-,W} \\
	\end{bmatrix}\begin{bsmallmatrix}
		(I_{n_+}+\Sigma\Sigma^{\HH})^{1/2} & \Sigma\\
		\Sigma^{\HH} & (I_{n_-}+\Sigma^{\HH}\Sigma)^{1/2}\\
	\end{bsmallmatrix})
	\\&=\trace(\begin{bsmallmatrix}
		\wtd \Omega_{+,V}(I_{n_+}+\Sigma\Sigma^{\HH})^{1/2} & \wtd \Omega_{+,V}\Sigma\\
		0 & 0\\
	\end{bsmallmatrix}\begin{bsmallmatrix}
		\Lambda_{+,W}(I_{n_+}+\Sigma\Sigma^{\HH})^{1/2} & \Lambda_{+,W}\Sigma\\
		\Lambda_{-,W}\Sigma^{\HH} & \Lambda_{-,W}(I_{n_-}+\Sigma^{\HH}\Sigma)^{1/2}\\
	\end{bsmallmatrix})
	\\&=\trace(
	\wtd \Omega_{+,V}\big[(I_{n_+}+\Sigma\Sigma^{\HH})^{1/2}\Lambda_{+,W}(I_{n_+}+\Sigma\Sigma^{\HH})^{1/2}
            - \Sigma\Lambda_{-,W}\Sigma^{\HH}\big]
		).
\end{align*}
The last infimum in \eqref{eq:min-trace-Herm2'':+} becomes
\begin{multline}\label{eq:main-SPD:min:sep:+:pf-3}
    \inf_{\Sigma_0\succeq0~\text{diagonal}\atop V_+,V_-,W_+,W_-~\text{unitary}}\trace(\wtd \Omega_VS\Lambda_W S)
	\\
  =\inf_{\Sigma_0\succeq0~\text{diagonal}\atop V_+,W_+,W_-~\text{unitary}}\trace(
	\wtd \Omega_{+,V}[(I_{n_+}+\Sigma\Sigma^{\HH})^{1/2}\Lambda_{+,W}(I_{n_+}+\Sigma\Sigma^{\HH})^{1/2}
              - \Sigma\Lambda_{-,W}\Sigma^{\HH}]
		)
		.
\end{multline}
This infimum is $-\infty$ if $\Omega \not\succeq0$.
In fact,
suppose $\wtd \Omega_+ x_+=\omega_k x_+$ where $\omega_k<0$, and $x_+$ is a unit eigenvector.
Construct
$
\what V_+=\begin{bmatrix}
	x_+ & V_{+,c}
	\end{bmatrix}\in\bbC^{n_+\times n_+}$ that is
unitary.
Thus, upon restrictions $V_+=\what V_+,\,W_+=I,\,W_-=I,\,\Sigma_0=\diag(\sigma,0,\dots,0)$, we have by
\eqref{eq:main-SPD:min:sep:+:pf-3}
\begin{align*}
\inf_{\Sigma_0\succeq0~\text{diagonal}\atop V_+,V_-,W_+,W_-~\text{unitary}}\trace(\wtd \Omega_VS\Lambda_W S)
&\le \inf_{\sigma}\trace(
	\what V_+^{\HH} \wtd\Omega_+ \what V_+[
\Lambda_++\sigma^2 (\lambda_{n_+}^+-\lambda_{n_-}^-)e_1e_1^{\HH}
	])
	\\&= \inf_{\sigma}\sigma^2 (\lambda_{n_+}^+-\lambda_{n_-}^-)\trace( \what V_+^{\HH}\wtd \Omega _+\what V_+ e_1e_1^{\HH})
	+\trace( \what V_+^{\HH}\wtd \Omega _+\what V_+ \Lambda_+)
	\\&= \inf_{\sigma}\sigma^2 (\lambda_{n_+}^+-\lambda_{n_-}^-)(e_1^{\HH} \what V_+^{\HH}\wtd \Omega _+\what V_+ e_1)
	+\trace( \what V_+^{\HH}\wtd \Omega _+\what V_+ \Lambda_+)
	\\&= \inf_{\sigma}\sigma^2 (\lambda_{n_+}^+-\lambda_{n_-}^-)(x_+^{\HH} \wtd \Omega _+x_+)
	+\trace( \what V_+^{\HH}\wtd \Omega _+\what V_+ \Lambda_+)
	\\&= \inf_{\sigma}\sigma^2 (\lambda_{n_+}^+-\lambda_{n_-}^-)\omega_k
	+\trace( \what V_+^{\HH}\wtd \Omega _+\what V_+ \Lambda_+)
	\\&=-\infty,
\end{align*}
as long as $\Lambda_+\ne0$ or $\Lambda_-\ne0$, which is equivalent to $A\ne 0$, where $e_1$ is the first column of the identity matrix.

So far, we have shown that if $\Omega \not\succeq0,A\ne 0$, the infimum is $-\infty$, for any positive semi-definite pencil $A-\lambda B$ with $B$ genuinely indefinite, except the case $A-\lambda B$ is not diagonalizable, to which we will return.

In what follows, suppose that $\Omega \succeq0$.
Then
\begin{align}
&\inf_{\Sigma_0\succeq0~\text{diagonal}\atop V_+,V_-,W_+,W_-~\text{unitary}}
	\trace(\wtd \Omega _{+,V}[(I+\Sigma \Sigma^{\HH})^{1/2}\Lambda_{+,W}(I+\Sigma \Sigma^{\HH})^{1/2}-\Sigma \Lambda_{-,W}\Sigma^{\HH}]) \nonumber \\
&=\inf_{V_{\pm},\,W_{\pm}~\text{unitary}}\inf_{\Sigma_0\succeq0~\text{diagonal}}
	\trace(\wtd \Omega _{+,V}[(I+\Sigma \Sigma^{\HH})^{1/2}\Lambda_{+,W}(I+\Sigma \Sigma^{\HH})^{1/2}
           -\Sigma \Lambda_{-,W}\Sigma^{\HH}])
	.  \label{eq:main-SPD:min:sep:+:pf-4}
\end{align}
Note that $\wtd \Omega _{+,V},\Lambda_{+,W},-\Lambda_{-,W}$ are positive semi-definite.
In the following, we will show that the ``infimum'' in \eqref{eq:main-SPD:min:sep:+:pf-4} is
$\sum_{i=1}^{k}\omega_i\lambda_i^+$ and is attained at $\Sigma=0$.

Firstly, since $\wtd \Omega _{+,V}\succeq 0$ and $\Sigma(-\Lambda_{-,W})\Sigma^{\HH}\succeq 0$,
we have by Theorem~\ref{thm:main-SPD:min}
\begin{equation}\label{eq:main-SPD:min:sep:+:pf-4a}
\trace(-\wtd \Omega _{+,V}\Sigma\Lambda_{-,W}\Sigma^{\HH})
    =\trace(\Sigma(-\Lambda_{-,W})\Sigma^{\HH})\ge 0.
\end{equation}
Secondly, we claim that
\begin{equation}\label{eq:main-SPD:min:sep:+:pf-4b}
\trace\big(\wtd \Omega _{+,V}(I+\Sigma \Sigma^{\HH})^{1/2}\Lambda_{+,W}(I+\Sigma \Sigma^{\HH})^{1/2}\big)
   \ge\sum_{i=1}^{k}\omega_i\lambda_i^+.
\end{equation}
Denote the eigenvalues of $(I+\Sigma \Sigma^{\HH})^{1/2}\Lambda_{+,W}(I+\Sigma \Sigma^{\HH})^{1/2}$ by
$\delta_1\le\delta_2\le\cdots\le\delta_{n_+}$. By Ostrowski's theorem \cite[p.283]{hojo:2013}, we known
\begin{equation}\label{eq:main-SPD:min:sep:+:pf-4c}
\lambda_i^+\le\delta_i\le (1+\|\Sigma\|_2^2)\lambda_i^+\quad\mbox{for $1\le i\le n_+$},
\end{equation}
where $\|\Sigma\|_2$ is the spectral norm of $\Sigma$.
Let $I_{n_+;k}=\begin{bmatrix}
	I_k\\ 0
\end{bmatrix}\in\bbC^{n_+\times k}$. We have
\begin{align}
\trace\big(\wtd \Omega _{+,V}&(I+\Sigma \Sigma^{\HH})^{1/2}\Lambda_{+,W}(I+\Sigma \Sigma^{\HH})^{1/2}\big)\nonumber \\
  &=\trace\big(V_+^{\HH}I_{n_+;k}\Omega I_{n_+;k}^{\HH}V_+(I+\Sigma \Sigma^{\HH})^{1/2}\Lambda_{+,W}(I+\Sigma \Sigma^{\HH})^{1/2}\big)\nonumber \\
  &=\trace\big(\Omega (V_+^{\HH}I_{n_+;k})^{\HH}(I+\Sigma \Sigma^{\HH})^{1/2}\Lambda_{+,W}(I+\Sigma \Sigma^{\HH})^{1/2}(V_+^{\HH}I_{n_+;k})\big)\nonumber \\
  &\ge\sum_{i=1}^{k}\omega_i\delta_i \qquad\mbox{(by Theorem~\ref{thm:main-SPD:min})}\nonumber \\
  &\ge\sum_{i=1}^{k}\omega_i\lambda_i^+, \label{eq:main-SPD:min:sep:+:pf-4d}
\end{align}
where we have used \eqref{eq:main-SPD:min:sep:+:pf-4c} in the last step. This is \eqref{eq:main-SPD:min:sep:+:pf-4b}.
It is not hard to see that the equality in \eqref{eq:main-SPD:min:sep:+:pf-4d} is attained at $\Sigma=0$ and
appropriately chosen $V_+$ and $W_+$.
Combining \eqref{eq:main-SPD:min:sep:+:pf-3}, \eqref{eq:main-SPD:min:sep:+:pf-4},
\eqref{eq:main-SPD:min:sep:+:pf-4a}, and \eqref{eq:main-SPD:min:sep:+:pf-4b}
completes the proof of the theorem for the case when
 $A-\lambda B$ is diagonalizable.

Consider now that $m_0>0$, namely \textbf{the pencil $A-\lambda B$ is not diagonalizable}.
We perturb $A-\lambda B$ to  $(A+\varepsilon E)-\lambda B$ with $\varepsilon>0$, where
\[
	E=U^{-\HH}\diag(0,0,E_{\rm b},0)U^{-1},\quad E_{\rm b}=\diag(\underbrace{E_0,\dots,E_0}_{m_0}),\quad E_0=\begin{bmatrix}
		1&0\\0&0
	\end{bmatrix}.
\]
Clearly $(A+\varepsilon E)\succeq0$ and the pencil is diagonalizable. Letting $\varepsilon\to 0^+$ leads to the desired result, based on the case for diagonalizable $A-\lambda B$.
\end{proof}

\begin{proof}[Proof of Corollary~\ref{cor:main-SPD:min:sep}]
Partition $X=\begin{bmatrix}
	X_+ & X_-
\end{bmatrix}$, and let $D_{\pm}=Q_{\pm}\Omega_{\pm}Q_{\pm}^{\HH}$ be the eigen-decompositions of $D_+$ and $D_-$, respectively.
First consider the case that $A-\lambda B$ is diagonalizable.
We have
\begin{align*}
\min_{X^{\HH}BX=J_k}\trace(DX^{\HH}AX)
&=\min_{X_{\pm}^{\HH}BX_{\pm}=\pm I_{k_{\pm}}\atop X_+^{\HH}BX_-=0}\trace(D_+X_+^{\HH}AX_++D_-X_-^{\HH}AX_-) \\
&\ge\min_{X_+^{\HH}BX_+=I_{k_+}}\trace(D_+X_+^{\HH}AX_+)+
     \min_{X_-^{\HH}BX_-=-I_{k_-}}\trace(D_-X_-^{\HH}AX_-) \\
  &=\sum_{i=1}^{k_+}\omega_i^+\lambda_i^+-\sum_{i=1}^{k_-}\omega_i^-\lambda_i^-,
\end{align*}
of which the last equality holds by making the columns of $X_{\pm}Q_{\pm}$ be the eigenvectors of $A-\lambda B$ associated with its eigenvalues $\lambda_i^{\pm}$ for $1\le i\le k_{\pm}$, respectively.
This proves \eqref{eq:main-SPD:min:sep}.

For the case that $A-\lambda B$ is not diagonalizable, \eqref{eq:main-SPD:min:sep} also holds, by using the same technique at the end of the proof of Theorem~\ref{thm:main-SPD:min:sep:+} above.
\end{proof}

\section{Conclusion}\label{sec:concl}
Previously, the classical Ky Fan's trace minimization principle on
$\min_{X^{\HH}X=I_k}\trace(X^{\HH}AX)$ for a Hermitian matrix $A$ has been extended
to about
\begin{align*}
&\min_{X^{\HH}BX=I_k}\trace(X^{\HH}AX)\quad\mbox{for positive definite $B$, or more generally} \\
&\inf_{X^{\HH}BX=J_k}\trace(X^{\HH}AX)\quad\mbox{for genuinely indefinite $B$},
\end{align*}
where $J_k$ is diagonal with diagonal entries $\pm 1$.
The extension for a positive definite $B$ is rather straightforward, but quite complicated
when $B$ is genuinely indefinite  \cite{kove:1995,lilb:2013}. In fact, the infimum can be $-\infty$ for the last case.

Our extensions in this paper
are along the line of the Brockett cost function:
$\trace(DX^{\HH}AX)$ in $X$ satisfying $X^{\HH}X=I_k$, when $D$ is Hermitian and positive
semi-definite. Specifically, we present elegant analytic solutions, in terms of
eigenvalues and eigenvectors of matrix pencil $A-\lambda B$, to
\begin{align}
&\min_{X^{\HH}BX=I_k}\trace(DX^{\HH}AX)\quad\mbox{for positive definite $B$}, \\
&\inf_{X^{\HH}BX=J_k}\trace(DX^{\HH}AX)\quad\mbox{for genuinely indefinite $B$}, \label{eq:XBX=J}
\end{align}
where $D$ is no longer assumed to be positive semi-definite. It is shown that the infimum in \eqref{eq:XBX=J} is finite if and only if $D$ is positive semi-definite.
Our analytic solutions are concise and our algebraic technique compares favorably to previously laborious effort for the case $B=I$ via the usual optimization technique \cite{lisw:2019}.

{\small
\def\noopsort#1{}\def\l{\char32l}\def\v#1{{\accent20 #1}}
  \let\^^_=\v\def\hbk{hardback}\def\pbk{paperback}

}

\end{document}